\documentclass{article}

\usepackage{amsthm}
\usepackage{amsfonts}
\usepackage{amssymb}
\usepackage{amsgen}
\usepackage{amsmath}
\usepackage{amsopn}
\usepackage{verbatim}
\usepackage{xypic}
\usepackage{xspace}
\usepackage{multicol}
\usepackage{makeidx}
\usepackage{eepic}
\usepackage{upref}

\theoremstyle{plain}
\newtheorem{thm}{Theorem}[section]
\newtheorem{lem}[thm]{Lemma}
\newtheorem{prop}[thm]{Proposition}

\theoremstyle{definition}

\newtheorem{defn}[thm]{Definition}

 \font\cyr=wncyr10
 \newcommand{\nc}{\newcommand}
 \nc{\reg}{{\rm reg}}
 \nc{\mal}{{{\scriptstyle \maltese}}}
 \nc{\re}{{\Re}}
 \nc{\fA}{{\mathfrak A}}
 \nc{\ra}{\rightarrow}
 \nc{\ors}{{\vec s\,}}
 \nc{\os}{{\overset}}
 \nc{\Z}{{\mathbb Z}}
 \nc{\R}{{\mathbb R}}
 \nc{\N}{{\mathbb N}}
 \nc{\ZN}{{\mathbb Z_{\ge 0}}}
 \nc{\Q}{{\mathbb Q}}
 \nc{\C}{{\mathbb C}}
 \nc{\D}{{\mathcal D}}
 \nc{\caT}{{\mathcal T}}
 \nc{\tB}{{\tilde B}}
 \nc{\Li}{{\rm Li}}
 \nc{\suf}{{\ast\,}}
 \nc{\sufq}{{\ast_q\,}}
 \nc{\gam}{{\gamma}}
 \nc{\ga}{{\alpha}}
 \nc{\gl}{{\lambda}}
 \nc{\gb}{{\beta}}
 \nc{\gd}{{\delta}}
 \nc{\gs}{{\sigma}}
 \nc{\gS}{{\Sigma}}
 \nc{\sif}{{\mathcal S}}
 \nc{\gt}{{\tau}}
 \nc{\Lra}{\Longrightarrow}
 \nc{\lra}{\longrightarrow}
 \nc{\fS}{{\mathfrak S}}
 \nc{\DD}{{\mathfrak D}}
 \nc{\Llra}{\Longleftrightarrow}
 \nc{\ol}{\overline}
 \nc{\lms}{\longmapsto}
 \nc{\zq}{{\zeta_q}}
 \nc\qup{{q\uparrow 1}}
 \nc{\us}{\underset}
 \nc{\tn}{{\tilde{n}}}
 \nc{\gD}{{\Delta}}
 \nc{\bk}{{\bf k}}

\nc{\sha}{{\mbox{\cyr x}}}

\begin{document}

\title{An exotic shuffle relation of  $\zeta(\{2\}^m)$ and $\zeta(\{3,1\}^n)$}
\author{Jianqiang Zhao}
\date{}
\maketitle
\begin{center}
{\large Department of Mathematics, Eckerd College, St. Petersburg,
FL 33711}
\end{center}
\vskip0.6cm

\begin{center}
{\bf Abstract} \end{center} In this short note we will provide a
new and shorter proof of the following exotic shuffle relation of
multiple zeta values:
 $$\zeta(\{2\}^m \sha\{3,1\}^n)={2n+m\choose m}
 \frac{\pi^{4n+2m}}{(2n+1)\cdot (4n+2m+1)!}.$$
This was proved by Zagier when $n=0$, by Broadhurst when $m=0$,
and by Borwein, Bradley, and Broadhurst when $m=1$. In general
this was proved by Bowman and Bradley in \emph{The algebra and
combinatorics of shuffles and multiple zeta values}, J. of
Combinatorial Theory, Series A, Vol. \textbf{97} (1)(2002),
43--63. Our idea in the general case is to use the method of
Borwein et al.~to reduce the above general relation to some
families of combinatorial identities which can be verified by
WZ-method.

\section{Introduction}
The multiple zeta values (MZVs) are defined by the series
\begin{equation}\label{zeta}
\zeta(s_1,\dots, s_d)=\sum_{k_1>\dots>k_d>0}
 \frac{1}{k_1^{s_1}\cdots k_d^{s_d}}
\end{equation}
where $s_1,\dots, s_d$ are positive integers and $s_1>1$. These
values  were first studied by Euler systematically but then were
forgotten for many years. In the past two decades they have become
quite popular research subjects due to their prominent roles in
many branches of mathematics and physics (see for example
\cite{BBB,H-Pac,Zag} and their references).

Two different kinds of shuffle products are the keys in
discovering linear relations among MZVs of the same weight. One
comes from the series representation given by \eqref{zeta} while
the other from iterated integral representations first noticed by
Kontsevich \cite{K1} when he was studying knot invariants.

In \cite{Zag} Zagier proposed the following conjecture which was
proved subsequently by Broadhurst (see \cite{BBB}):
\begin{equation}\label{equ:31}
\zeta(\{3,1\}^n) = \frac{2\pi^{4n}}{(4n+2)!}.
\end{equation}
Then Borwein, Bradley and Broadhurst \cite{BBB} proved the
following exotic shuffle formula:
 $$\zeta(\{2\}\sha\{3,1\}^n ) =\frac{\pi^{4n+2}}{(4n+3)!}.$$
This was subsequently generalized by Bowman and Bradley as
follows:
\begin{thm} \emph{(\cite[Corollary 5.1]{BoB}) }\label{thm:main} For any positive integers $m$ and $n$,
 $$\zeta(\{2\}^m \sha\{3,1\}^n)= {2n+m\choose m}
 \frac{\pi^{4n+2m}}{(2n+1)\cdot (4n+2m+1)!}.$$
\end{thm}
In this short note we shall provide a new and shorter proof of
this relation.

This paper was conceived while I was visiting Chern Institute of
Mathematics at Nankai University and written at the Morningside
Center of Mathematics of Academia Sinica at Beijing in the summer
of 2007. I would like to thank both institutions and my hosts
Chengming Bai and Fei Xu for their hospitality and the ideal
working environment. Thanks also go to Doron Zeilberger who kindly
suggested an approach to proving the key combinatorial identities
in the paper.

\section{Some known  results}
We first recall some notation and results in \cite{BBB}. Let
$A=dx/x$ and $B=dx/(1-x)$ and form the alphabet $\{A,B\}$. Then
every MZV can be expressed as an iterated integral $\int_0^1 w$
where $w$ is some word beginning with $A$ and ending with $B$. For
example, $\zeta(3,1)=\int_0^1 A^2B^2$. Using Chen's theory of
iterated integrals \cite{Chen} we have
 $$\int_0^1 f_1\cdots f_m \int_0^1 f_{m+1}\cdots f_{m+n}
  =\int_0^1 (f_1\cdots f_m )\sha (f_{m+1}\cdots f_{m+n})$$
for one forms $f_1,\dots,f_{m+n}$, where
 $$(f_1\cdots f_m )\sha (f_{m+1}\cdots f_{m+n})=\sum_{\substack{\gs\in {\mathfrak
 S}_{m+n},
 \sigma^{-1}(1)<\cdots<\sigma^{-1}(n)\\  \sigma^{-1}(n+1)<\cdots< \sigma^{-1}(n+m)}}
 f_{\gs(1)}\dots f_{\gs(m+n)} $$
 is the shuffle product. Here ${\mathfrak S}_{m+n}$ is the permutation group of $m+n$ letters.

\begin{defn}
\label{def-spqi} (\cite[Defn.~1 and Defn.~2]{BBB}) Let $p$, $q$
and $j$ be non-negative integers such that ${\rm min}(p,q)\ge j$.
Let $T_{p+q,j}$ denote the sum of all those words occurring in
$(AB)^p\sha (AB)^q$ that contain the subword $A^2$ exactly $j$
times.
\end{defn}
We now provide two lemmas for the proof of the main result
Theorem~\ref{thm:main}. The first one is due to Borwein, Bradley
and Broadhurst.
\begin{lem} \label{lem:prop} \emph{(\cite[Prop.~1]{BBB})}
\label{ABp-ABq} For any non-negative integers $p$ and $q$ we have
\[
(AB)^p \sha (AB)^q = \sum_{j=0}^{\min(p,q)}
{{p+q-2j}\choose{p-j}}\cdot 4^j T_{p+q,j}.
\]
\end{lem}

\begin{lem}\label{lem:keylem}
\label{lem:key} For any positive integers $n\ge m$ we have
 \begin{align}
 \label{Todd}
\sum_{k=1-n}^n (-1)^k k^{2m-1} \left[ (AB)^{n-k} \sha
(AB)^{n-1+k}\right] =&4^{n-j}\sum_{j=1}^{m}
 T_{2n-1,n-j}\cdot a_{2m-1,j},\\
 \label{Teven}
\sum_{k=-n}^n (-1)^k k^{2m} \left[ (AB)^{n-k} \sha (AB)^{n+k}
\right] =&4^{n-j} \sum_{j=1}^{m}
 T_{2n,n-j}\cdot a_{2m,j},
\end{align}
where for $j=1,\dots,m$ we set
  \begin{equation}\label{equ:aodd}
a_{2m-1,j}=\sum_{k=1-j}^j (-1)^k k^{2m-1}{2j-1\choose
 j-k},\qquad
 a_{2m,j}=\sum_{k=-j}^j (-1)^k k^{2m}{2j\choose j-k}.
\end{equation}
\end{lem}
\begin{proof} By Lemma \ref{lem:prop} the left hand side of
\eqref{Todd} can be written as
$$\left(\sum_{k=1-n}^0 \sum_{j=0}^{n-1+k}+
 \sum_{k=1}^n \sum_{j=0}^{n-k} \right)(-1)^k k^{2m-1}
{{2n-1-2j}\choose{n-k-j}}\cdot  4^j T_{2n-1,j}$$ which after
reordering is equal to
 $$\sum_{j=0}^{n-1}4^j
 T_{2n-1,j}\cdot \sum_{k=1-n+j}^{n-j}(-1)^k
 k^{2m-1}{{2n-1-2j}\choose{n-k-j}}.$$
Then from \eqref{equ:aodd} (with $n-j$ in the place of $j$)
 $$\sum_{k=1-n}^n (-1)^k k^{2m-1} \left[ (AB)^{n-k} \sha
(AB)^{n-1+k} \right] =4^{n-j}\sum_{j=n-m}^{n-1}
 T_{2n-1,j}\cdot a_{2m-1,n-j}.$$
The proof of \eqref{Teven} is similar and is left to the
interested readers.
\end{proof}

\section{Proof of Theorem~\ref{thm:main}}
The special case of Theorem~\ref{thm:main} when $n=0$ was proved
in \cite{H-Pac,BBB}: for every positive integer $r$
\begin{equation}\label{equ:z22}
 \zeta(\{2\}^r)=  \frac{\pi^{2r}}{(2r+1)!}.
\end{equation}
Applying $\int_0^1$ to the two equations in Lemma \ref{lem:keylem}
and using equation \eqref{equ:z22} we see that
\begin{align} \label{equ:xodd}
 \sum_{k=1-n}^n \frac{(-1)^k k^{2m-1}
 \pi^{4n-2}}{(2n+2k-1)!(2n-2k+1)!} = &  4^n\sum_{j=1}^{m}
 x_{2n-1,n-j}\cdot a_{2m-1,j},\\
 \sum_{k=-n}^n \frac{(-1)^k k^{2m}
\pi^{4n}}{(2n+2k+1)!(2n-2k+1)!} = &4^n\sum_{j=1}^{m}
 x_{2n,n-j}\cdot a_{2m,j},  \label{equ:xeven}
\end{align}
where
$$x_{2n-1,n-j}=\int_0^1 T_{2n-1,n-j},\qquad
  x_{2n,n-j}=\int_0^1  T_{2n,n-j}.$$
It is straightforward to see that
$$x_{2n-1,n-j}=\zeta(\{2\}^{2j-1}\sha\{3,1\}^{n-j}),\qquad
  x_{2n,n-j}=\zeta(\{2\}^{2j}\sha\{3,1\}^{n-j}).$$
Note that for arbitrary fixed $n$, $x_{2n-1,n-j}$ (resp.
$x_{2n,n-j}$) are recursively defined by \eqref{equ:xodd} (resp.
\eqref{equ:xeven}) when we take $m=1,\dots,n-1$. Thus
Theorem~\ref{thm:main} is quickly reduced to the following
combinatorial identities: For all $m,n\ge 1$ we have
\begin{align}\label{equ:genodd}
 \sum_{k=1-n}^n \frac{(-1)^k k^{2m-1}}{2n-2k+1} {4n-1\choose 2n+2k-1}=& \sum_{j=1}^{m}\sum_{k=1-j}^j
   \frac{4^{n-j}(-1)^k k^{2m-1} }{2n-2j+1}   {2n-1\choose 2j-1} {2j-1\choose j-k},\\
 \sum_{k=-n}^n \frac{(-1)^k k^{2m}}{2n-2k+1} {4n+1\choose 2n+2k+1} = & \sum_{j=1}^{m}\sum_{k=-j}^j
  \frac{4^{n-j}(-1)^k k^{2m}}{2n-2j+1} {2n\choose 2j} {2j\choose j-k}.\label{equ:geneven}
\end{align}
Now observe that we may replace the outer sum of the right hand
side in both \eqref{equ:genodd} and \eqref{equ:geneven} by
$\sum_{j=1}^\infty$ by the following computation: if $j>m$ then
 \begin{align*}
 \sum_{k=1-j}^j  (-1)^k k^{2m-1}
{2j-1\choose j-k}=&\left(x\frac{d}{dx}\right)^{2m-1}
\sum_{k=1-j}^j
 (-1)^k x^k {2j-1\choose j-k}\Big|_{x=1}\\
 =&(-1)^{1-j}\left(x\frac{d}{dx}\right)^{2m-1}  x^{1-j}
 (1-x)^{2j-1}\Big|_{x=1}=0.
\end{align*}
Similarly, if $j>m$ then
$$ \sum_{k=-j}^j
  (-1)^k k^{2m} {2j\choose j-k}=0.$$
We therefore only need to prove that for all $m,n\ge 1$
\begin{align}\label{equ:genodd1}
 \sum_{k=1-n}^n \frac{(-1)^k k^m}{2n-2k+1} {4n-1\choose 2n+2k-1} =& \sum_{j=1}^\infty \sum_{k=1-j}^j
   \frac{4^{n-j}(-1)^k k^m }{2n-2j+1} {2n-1\choose 2j-1} {2j-1\choose j-k},\\
 \sum_{k=-n}^n \frac{(-1)^k k^m}{2n-2k+1} {4n+1\choose 2n+2k+1} = & \sum_{j=1}^\infty  \sum_{k=-j}^j
  \frac{4^{n-j}(-1)^k k^m}{2n-2j+1} {2n\choose 2j} {2j\choose
  j-k}.\label{equ:geneven1}
\end{align}
In fact, we only need the case $m$ is odd in \eqref{equ:genodd1}
and $m$ is even in \eqref{equ:geneven1}.

The following crucial step is suggested by D.~Zeilberger to whom
we are very grateful. It is a well-known fact that $x\Q[x]$ have
two bases over $\Q$: $\{x^m: m\ge 1\}$ and $\{{x \choose m}:m\ge
1\}$. So Theorem~\ref{thm:main} follows immediately from the
following proposition. Note that we will exchange the index $j$
and $k$ on the right hand side.

\begin{prop} For all $m,n\ge 1$ we have
\begin{align}\label{equ:oddcase}
 \sum_{k=1-n}^n \frac{(-1)^k }{2n-2k+1} {k \choose m}{4n-1\choose 2n+2k-1} =& \sum_{k=1}^\infty
 \sum_{j=1-k}^k
   \frac{4^{n-k}(-1)^j  }{2n-2k+1}{j \choose m}   {2n-1\choose 2k-1} {2k-1\choose k-j},\\
 \sum_{k=-n}^n \frac{(-1)^k }{2n-2k+1} {k \choose m}{4n+1\choose 2n+2k+1}= & \sum_{k=1}^\infty
 \sum_{j=-k}^k
  \frac{4^{n-k}(-1)^j}{2n-2k+1}{j \choose m} {2n\choose 2k} {2k\choose
  k-j}.\label{equ:evencase}
\end{align}
\end{prop}
\begin{proof} We first break the inner sum on the right hand side
of \eqref{equ:oddcase} as follows:
 $$\sum_{j=1-k}^k (-1)^j {j \choose m}  {2k-1\choose k-j}=
 \sum_{j=1}^k (-1)^j {j \choose m}  {2k-1\choose k-j}
 + \sum_{j=1}^{k-1} (-1)^{m-j} {m+j-1 \choose m}  {2k-1\choose
 j+k}.$$
Reindexing and using the Chu-Vandermonde identity  (see
\cite[p.~182]{A=B}) we change the above to
\begin{align*}
\ & \sum_{j=0}^{k-m} (-1)^{m} {-m-1\choose j}  {2k-1\choose k-m-j}
 + \sum_{j=0}^{k-2} (-1)^{m-1} {-m-1\choose j}  {2k-1\choose k-2-j}\\
 =&    (-1)^{m}  {2k-m-2\choose k-m} +(-1)^{m-1} {2k-m-2\choose
 k-2}.
\end{align*}
Notice that the sum is 0 when $k\ge m$, however, when $k<m$ only
the first term in the sum is always 0. By denoting the left (resp.
right) hand side of \eqref{equ:oddcase} by $L(m,n)$ (resp.
$R(m,n)$) we get
\begin{equation}
\label{equ:Rodd} R(m,n)= \sum_{k=1}^{n}
 \frac{(-1)^{m} 4^{n-k}}{(2n-2k+1)}{2n-1\choose
 2k-1}  \left({2k-m-2\choose k-m}- {2k-m-2\choose k-2}\right).
\end{equation}
Set
 \begin{align*}
F_0(n,k):=&\frac{(-1)^k}{ 2n-2k+1 } {k \choose m}{4n-1\choose
2n+2k-1}, \\
 F_1(n,k):=&\frac{(-1)^{m+1} 4^{n-k}}{ 2n-2k+1}{2n-1\choose
 2k-1} {2k-m-2\choose k-2},\\
 F_2(n,k):=&\frac{(-1)^{m} 4^{n-k}}{ 2n-2k+1}{2n-1\choose
 2k-1} {2k-m-2\choose k-m}.
\end{align*}
Then clearly $\lim_{k\to \pm \infty}F_i(n,k)=0$ for each $n$ and
$$L(m,n)=\sum_{k\in \Z} F_0(n,k),\quad R(m,n)= \sum_{k\in \Z}
 F_1(n,k)+ F_2(n,k).$$
Now by WZ-method  (see \cite{A=B}) we can find functions
$G_i(n,k)$ ($i=0,1,2$) and coefficients $c_j$ ($j=0,1,2,3$) such
that all the three functions $F_i$ above satisfy the following:
for all $k$
  \begin{multline}
 \label{equ:Gs}
 c_3F_i(n+3,k)-c_2F_i(n+2,k)+c_1F_i(n+1,k)-c_0F_i(n,k)=G_i(n,k+1)-G_i(n,k).
 \end{multline}
Now summing up \eqref{equ:Gs} for $-\infty <k<\infty$ we find that
both $L(m,n)$ and $R(m,n)$ satisfy the same recurrence relation
 $$c_3A(n+3)=c_2A(n+2)-c_1A(n+1)+c_0A(n).$$
Furthermore it is easy to check by using \eqref{equ:oddcase} that
for all $m\ge 1$ ($\gd_{m,1}$ is the Kronecker symbol)
$$L(m,1)=R(m,1)=-\gd_{m,1},\quad L(m,2)=R(m,2)=-5\gd_{m,1}-(-1)^m.$$
When $n=3$ direct computation using \eqref{equ:Rodd} shows that
 $$L(m,3)=R(m,3)=\begin{cases}
 -16 &\text{if } m=1,\\
 0  &\text{if } m=2,\\
 40/3 &\text{if } m=3,\\
 (-1)^m(m-52/3) \quad&\text{if } m\ge 4.\end{cases} $$
This implies that \eqref{equ:oddcase} is true.

\medskip
The proof of \eqref{equ:evencase} can be carried out by exactly
the same method as above so we leave the details to the interested
readers. This completes the proof of the proposition and therefore
Theorem~\ref{thm:main}.
\end{proof}

\noindent {\em Email:} zhaoj@eckerd.edu

\end{document}